\documentclass[runningheads]{llncs}
\usepackage{amssymb,amsfonts,amsmath}
\usepackage{cite}
\usepackage{graphicx}
\usepackage{array}
\usepackage{algorithm}
\usepackage{algorithmicx}
\usepackage[noend]{algpseudocode}
\usepackage[verbose]{wrapfig}
\usepackage[colorlinks,linkcolor=cyan,citecolor=blue]{hyperref}
\usepackage{multirow}
\def\car{\circlearrowleft}

\def\cL{\mathcal{L}}
\def\cC{\mathcal{C}}
\def\cF{\mathcal{F}}
\def\create{\mathsf{create}}
\def\labl{\mathsf{label}}
\def\trim{\mathsf{trim}}

\begin{document}

\title{Distance Labeling for Families of Cycles\thanks{The first author is supported by the grant MPM no. ERC 683064 under the
 EU's Horizon 2020 Research and Innovation Programme and by the State of Israel through the Center for Absorption in Science of the Ministry of Aliyah and Immigration.}}

%


\author{Arseny M. Shur\inst{1} \and Mikhail Rubinchik\inst{2}}

\institute{Bar Ilan University, Ramat Gan, Israel\\ \email{shur@datalab.cs.biu.ac.il} \and
SPGuide Online School in Algorithms\\
\email{mikhail.rubinchik@gmail.com} }
\maketitle

\begin{abstract}
For an arbitrary finite family of graphs, the distance labeling problem asks to assign labels to all nodes of every graph in the family in a way that allows one to recover the distance between any two nodes of any graph from their labels.
The main goal is to minimize the number of unique labels used.
We study this problem for the families $\cC_n$ consisting of cycles of all lengths between 3 and $n$. 
We observe that the exact solution for directed cycles is straightforward and focus on the undirected case.
We design a labeling scheme requiring $\frac{n\sqrt{n}}{\sqrt{6}}+O(n)$ labels, which is almost twice less than is required by the earlier known scheme. 
Using the computer search, we find an optimal labeling for each $n\le 17$, showing that our scheme gives the results that are very close to the optimum.
\vspace*{-2mm}
\end{abstract}


\section{Introduction}
\emph{Graph labeling} is an important and active area in the theory of computing.
A typical problem involves a parametrized finite family $\cF_n$ of graphs (e.g., all planar graphs with $n$ nodes) and a natural function $f$ on nodes (e.g., distance for \emph{distance labeling} or  adjacency for \emph{adjacency labeling}).
The problem is to assign labels to all nodes of every graph in $\cF_n$ so that the function $f$ can be computed solely from the labels of its arguments.
Note that the algorithm computing $f$ knows $\cF_n$ but not a particular graph the nodes belong to.
The main goal is to minimize the number of distinct labels or, equivalently, the maximum length of a label in bits. 
Additional goals include the time complexity of computing both $f$ and the labeling function.
In this paper, we focus solely on the main goal.

\setcounter{footnote}{0} 
The area of graph labeling has a rather long history, which can be traced back at least to the papers \cite{Breuer66, BrFo67}. 
The main academic interest in this area is in finding the limits of efficient representation of information.
For example, the adjacency labeling of $\cF_n$ with the minimum number of labels allows one to build the smallest ``universal'' graph, containing all graphs from $\cF_n$ as induced subgraphs. 
Similarly, the optimal distance labeling of $\cF_n$ gives the smallest ``universal'' matrix, containing the distance matrices of all graphs from $\cF_n$ as principal minors. 
The distributed nature of labeling makes it also interesting for practical applications such as distributed data structures and search engines \cite{GaPe03, AAKMR06, CKM10}, routing protocols \cite{EGP03, ThZw01} and communication schemes \cite{Winkler83}. 
 
The term \emph{distance labeling} was coined by Peleg in 1999 \cite{Peleg99}, though some of the results are much older \cite{GrPo72, Winkler83}. 
Let us briefly recall some remarkable achievements.
For the family of all undirected graphs with $n$ nodes it is known that the labels of length at least $\frac{n}{2}$ bits are necessary \cite{Moon65, KNR92}. 
The first labeling scheme with $O(n)$-bit labels was obtained by Graham and Pollak \cite{GrPo72}. The state-of-the-art labeling by Alstrup et al. \cite{AAKMR06} uses labels of length $\frac{\log3}{2}n+o(n)$ bits\footnote{Throughout the paper, $\log$ stands for the binary logarithm.} and allows one to compute the distances in $O(1)$ time (assuming the word-RAM model).

For planar graphs with $n$ nodes, the lower bound of $\Omega(\sqrt[3]{n})$ bits per label and a scheme using $O(\sqrt{n}\log n)$ bits per label were presented in \cite{GPPR04}. Recently, Gawrychowski and Uznanski \cite{GaUz23} managed to shave the log factor from the upper bound. Some reasons why the gap between the lower and upper bounds is hard to close are discussed in \cite{AGMW18}.

For trees with $n$ nodes, Peleg \cite{Peleg99} presented  a scheme with $\Theta(\log^2 n)$-bits labels.  Gavoille et al. \cite{GPPR04} proved that $\frac{1}{8}\log^2 n$ bits are required and $\approx1.7\log^2 n$ bits suffice. Alstrup et al. \cite{AGHPo16} improved these bounds to $\frac{1}{4}\log^2 n$ and $\frac{1}{2}\log^2 n$ bits respectively. Finally, Freedman et al. \cite{FGNW17} reached the upper bound $(\frac{1}{4}+o(1))\log^2 n$, finalizing the asymptotics up to lower order terms.

Further, some important graph families require only polynomially many labels and thus $O(\log n)$ bits per label. Examples of such families include interval graphs \cite{GaPa08},  permutation graphs \cite{BaGa09}, caterpillars and weighted paths \cite{AGHPo16}.

Cycles are among the simplest graphs, so it may look surprising that graph labeling problems for the family $\cC_n$ of all cycles up to length $n$ are not settled yet. 
A recent result \cite{AAHKS20} states that $n+\sqrt[3]{n}$ labels are necessary and $n+\sqrt{n}$ labels are sufficient for the \emph{adjacency} labeling of $\cC_n$. Still, there is a gap between the lower and the upper bounds. (As in the case of planar graphs, this is the gap between $\sqrt[3]{n}$ and $\sqrt{n}$, though at a different level.)

To the best of our knowledge, no papers were published yet on the \emph{distance} labeling of the family $\cC_n$. 
However, there are some folklore results\footnote{Communicated to us by E. Porat.}, namely, a lower bound of $\Omega(n^{4/3})$ labels and a labeling scheme requiring $O(n^{3/2})$ labels; together they produce another gap between $\sqrt[3]{n}$ and $\sqrt{n}$.

In this paper we argue that the upper estimate is correct. We describe a distance labeling scheme for $\cC_n$ that uses almost twice less labels than the folklore scheme. 
While this is a rather small improvement, we conjecture that our scheme produces labelings that are optimal up to an additive $O(n)$ term. 
To support this conjecture, we find the optimal number of labels for the families $\cC_n$ up to $n=17$. 
Then we compare the results of our scheme with an extrapolation of the optimal results, demonstrating that the difference is a linear term with a small constant.
Finally, we describe several improvements to our scheme that further reduce this constant.


\section{Preliminaries}

Given two nodes $u,v$ in a graph, the \emph{distance} $d(u,v)$ is the length of the shortest $(u,v)$-path.
Suppose we are given a finite family $\cF=\{(V_1,E_1),\ldots,(V_t,E_t)\}$ of graphs and an arbitrary set $\cL$, the elements of which are called \emph{labels}.
A \emph{distance labeling} of $\cF$ is a function $\phi:V_1\cup\cdots\cup V_t\to \cL$ such that there exists a function $d': \cL^2\to\mathbb{Z}$ satisfying, for every $i$ and each $u,v\in V_i$, the equality $d'(\phi(u),\phi(v))=d(u,v)$.
Since $d(u,u)=0$, no label can appear in the same cycle twice.
Thus for every single graph in $\cF$ we view labels as unique names for nodes and identify each node with its label when speaking about paths, distances, etc.

In the rest of the paper, \emph{labeling} always means distance labeling.
A \emph{labeling scheme} (or just \emph{scheme}) for $\cF$ is an algorithm that assigns labels to all nodes of all graphs in $\cF$. 
A scheme is \emph{valid} if it outputs a distance labeling.

We write $C_n$ for the undirected cycle on $n$ nodes and let $\cC_n=\{C_3,\ldots, C_n\}$.
The family $\cC_n$ is the main object of study in this paper. We denote the minimum number of labels in a labeling of $\cC_n$ by $\lambda(n)$. 

\subsection{Warm-up: Labeling Directed Cycles}

Consider a distance labeling of the family $\cC_n^\car=\{C_3^\car,\ldots,C_n^\car\}$ of \emph{directed} cycles. 
Here, the distance between two vertices is the length of the \emph{unique} directed path between them. 
We write $\lambda_D(n)$ for the minimum number of labels needed to label $\cC_n^\car$. 
A bit surprisingly, the exact formula for $\lambda_D(n)$ can be easily found.

\begin{proposition}
    \label{p:directed}
    $\lambda_D(n)=\frac{n^2 +2n+ n\bmod 2}{4}$.
\end{proposition}
\begin{proof}
    Let $u$ and $v$ be two labels in $C_i^\car$. 
    Then $d(u,v)+d(v,u)=i$. 
    Hence $u$ and $v$ cannot appear together is any other cycle. 
    Thus every two cycles have at most one label in common.
    Then $C_{n-1}^\car$ contains at least $n-2$ labels unused for $C_n^\car$, $C_{n-2}^\car$ contains at least $n-4$ labels unused for both $C_n^\car$ and $C_{n-1}^\car$, and so on.
    This gives the total of at least $n+(n-2)+(n-4)+\cdots+ n\bmod 2\ $ labels, which sums up exactly to the stated formula. 
    To build a labeling with this number of labels, label cycles in decreasing order; labeling $C_i^\car$, reuse one label from each of the larger cycles such that neither label is reused twice.\qed
\end{proof}

\subsection{Basic Facts on Labeling Undirected Cycles}
\label{ss:basic}

\begin{wrapfigure}[12]{R}{4.2cm} 
    \vspace*{-9mm}
    \centering
    \includegraphics[scale=0.9, trim = 41 723 455 22, clip]{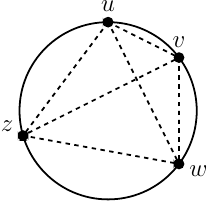}
    \caption{Triangles in a labeled cycle. The sets  $\{u,w,z\}$ and $\{v,w,z\}$ are triangles, while $\{u,v,w\}$ and $\{u,v,z\}$ are not.} 
    \label{f:triangles}
    \vspace*{-9mm}
\end{wrapfigure}
From now on, all cycles are undirected, so the distance between two nodes in a cycle is the length of the shortest of two paths between them. 
The maximal distance in $C_i$ is $\lfloor i/2\rfloor$. 
We often view a cycle as being drawn on a circumference, with equal distances between adjacent nodes, and appeal to geometric properties.

We say that a set $\{u,v,w\}$ of labels occurring in a cycle is a \emph{triangle} if each of the numbers $d(u,v),d(u,w),d(v,w)$ is strictly smaller than the sum of the other two. 
Note that three nodes are labeled by a triangle if and only if they form an acute triangle on a circumference (see Fig.~\ref{f:triangles}).
\begin{lemma} \label{l:triangle} 
    In a labeling of a family of cycles each triangle appears only once. 
\end{lemma}
\begin{proof}
    If $\{u,v,w\}$ is a triangle in $C_i$, then  $d(u,v)+d(u,w)+d(v,w)=i$. \qed
\end{proof}
Lemma~\ref{l:triangle} implies the only known lower bound on the number of labels for $\cC_n$.
\begin{proposition} \label{p:lower}
    Each labeling of $\cC_n$ contains $\Omega(n^{4/3})$ distinct labels.
\end{proposition}
\begin{proof}
    As $n$ tends to infinity, the probability that a random triple of labels of $C_n$ forms a triangle approaches the probability that three random points on a circumference generate an acute triangle. 
    The latter probability is $1/4$: this is a textbook exercise in geometric probability. 
    Thus, the set of labels of $C_n$ contains $\Omega(n^3)$ triangles. 
    By Lemma~\ref{l:triangle}, the whole set of labels contains $\Omega(n^4)$ distinct triangles; they contain $\Omega(n^{4/3})$ unique labels. \qed
\end{proof}

By \emph{diameter} of a cycle $C_i$ we mean not only the maximum length of a path between its nodes (i.e., $\lfloor i/2\rfloor$) but also any path of this length. From the context it is always clear whether we speak of a path or of a number.

\begin{lemma}\label{l:diam} 
    For any labeling of two distinct cycles $C_i$ and $C_j$ there exist a diameter of $C_i$ and a diameter of $C_j$ such that every label appearing in both cycles belongs to both these diameters.
\end{lemma}
\begin{proof}
    Let $i>j$ and let $\cL$ be the set of common labels of $C_i$ and $C_j$.
    We assume $\#\cL\ge3$ as otherwise the diameters trivially exist.
    By Lemma~\ref{l:triangle}, $\cL$ contains no triangles.
    Hence for every triple of elements of $\cL$ the maximum distance is the sum of the two other distances.
    Let $u,v$ be labels with the maximum distance in $\cL$.
    We have $d(u,v)\le \lceil i/2\rceil-1$, because there are no larger distances in $C_j$.
    Then the shortest $(u,v)$-path in $C_i$ is unique.
    Since $d(u,v)=d(u,w)+d(w,v)$ for any $w\in\cL$ by the maximality of $d(u,v)$, all labels from $\cL$ appear on this unique path, and hence on any diameter containing this path.

    Though $C_j$ may contain two shortest $(u,v)$-paths, all labels from $\cL$ appear on one of them. 
    Indeed, let $w,z\in\cL\setminus \{u,v\}$. Considering the shortest $(u,v)$-path in $C_i$, we have w.l.o.g. $d(u,w)+d(w,z)+d(z,v)=d(u,v)$. 
    This equality would be violated if $w$ and $z$ belong to different shortest $(u,v)$-paths in $C_j$. 
    Thus, $C_i$ has a diameter, containing the shortest $(u,v)$-path with all labels from $\cL$ on it. \qed
\end{proof}

\begin{wrapfigure}[12]{R}{5.3cm} 
    \vspace*{-7mm}
    \centering
    \includegraphics[scale=0.9, trim = 14 722 423 25, clip]{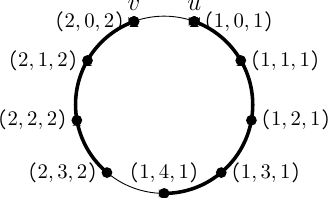}
    \caption{Folklore scheme: labeling $C_9$ in the family  $\cC_{10}$. One has $\lceil\sqrt{10}\rceil=4$, $m_1=9\bmod 4=1$, $m_2=\lfloor 9/4\rfloor=2$.}
    \label{f:folkscheme}
    \vspace*{-14mm}
\end{wrapfigure}
The known upper bound for cycle labeling is based on the following \emph{folklore labeling scheme}. 
For each $C_i\in\cC_n$ we choose arbitrary adjacent nodes $u$ and $v$ and cover $C_i$ by two disjoint paths: the path $P_1$ contains $\lceil i/2\rceil$ nodes, including $u$, while $P_2$ contains $\lfloor i/2 \rfloor$ nodes, including $v$ (see the example in Fig.~\ref{f:folkscheme}). 
Each node from $P_1$ gets the label $(1, d_1, m_1)$, where $d_1$ is the distance to $u$ and $m_1=i\bmod \big\lceil\sqrt{n}\big\rceil$; each node from $P_2$ gets the label $(2, d_2, m_2)$, where $d_2$ is the distance to $v$, $m_2=\big\lfloor i/\big\lceil\sqrt{n}\big\rceil\big\rfloor$.

\begin{proposition}[folklore] \label{p:folk_lab}
    The folklore scheme is valid and uses $\frac{3}{4}n\sqrt{n}+O(n)$ labels to label $\cC_n$.
\end{proposition}
\begin{proof}
    We prove the validity describing a procedure that derives the distance between two nodes from their labels (for illustration, see Fig.~\ref{f:folkscheme}).
    Suppose that the labels $(b,d,m)$ and $(b',d',m')$ appear together in some (unknown) cycle. 
    If $b=b'$, the two labels belong to the same path, so the distance between them is $|d-d'|$. 
    If $b\ne b'$, we compute the length $i$ of the cycle from the pair $(m, m')$. 
    The two analysed labels are connected by a path of length $d+d'+1$ and by another path of length $i-d-d'-1$; comparing these lengths, we get the distance.

    To compute the number of triples $(b,d,m)$ used for labeling, note that there are two options for $b$, $\lfloor n/2\rfloor+1$ options for $d$, and $\big\lceil\sqrt{n}\big\rceil$ options for $m$, for the total of $n\sqrt{n}+O(n)$ options. However, some triples are never used as labels. The label $(b,d,m)$ is unused iff for every number $i\le n$ compatible with $m$, the length of the path $P_b$ in $C_i$ is less than $d$.  If $b=1$, the maximum $i$ compatible with $m$ is $i= (\lceil\sqrt{n}\rceil-1)\lceil\sqrt{n}\rceil+m$, which is $O(\sqrt{n})$ away from $n$. Therefore, each of $O(\sqrt{n})$ values of $m$ gives $O(\sqrt{n})$ unused labels, for the total of $O(n)$. Let $b=2$. The maximum $i$ compatible with $m$ is $i= (m+1)\lceil\sqrt{n}\rceil-1$. The number of impossible values of $d$ for this $m$ is $(\lfloor n/2\rfloor-1)-(\lfloor i/2\rfloor-1)=\frac{\sqrt{n}-m}{2}\cdot\sqrt{n}+O(\sqrt{n})$. Summing up these numbers for $m=0,\ldots,\lceil\sqrt{n}\rceil-1$, we get $\frac{n\sqrt{n}}{4}+O(n)$ unused labels. In total we have  $\frac{n\sqrt{n}}{4}+O(n)$ unused labels out of $n\sqrt{n}+O(n)$ possible; the difference gives us exactly the stated bound. \qed
\end{proof}

\section{More Efficient Labeling Scheme and Its Analysis}

We start with more definitions related to labeled cycles. 
An \emph{arc} is a labeled path, including the cases of one-node and empty paths. 
The labels on the arc form a string (up to reversal), and we often identify the arc with this string.
In particular, we speak about \emph{substrings} and \emph{suffixes} of arcs.
``To label a path $P$ with an arc $a$'' means to assign labels to the nodes of $P$ to turn $P$ into a copy of $a$.

By \emph{intersection} of two labeled cycles we mean the labeled subgraph induced by all their common labels. Clearly, this subgraph is a collection of arcs. Lemma~\ref{l:diam} says that the intersection is a subgraph of some diameter of each cycle. The intersections of two arcs and of a cycle and an arc are defined in the same way.

An \emph{arc labeling} of a family of cycles is a labeling with the property that the intersection of any two cycles is an arc. Arc labelings are natural and easy to work with. Note that the folklore scheme produces an arc labeling: the intersection of two cycles is either empty or coincides with the path $P_1$ or the path $P_2$ of the smaller cycle. The schemes defined below produce arc labelings as well.

\paragraph{2-arc labeling scheme.} In this auxiliary scheme, cycles are labeled sequentially. 
The basic step is ``label a cycle with two arcs''.
Informally, we partition the cycle into two paths of equal or almost equal length and label each of them with an arc (or with its substring of appropriate length). To specify details and reduce ambiguity, we define this step in the function $\labl(a_1,a_2,C_j)$ below. 

\begin{algorithm*}[!htb]
\begin{algorithmic}[1]
    \State{\textbf{function} $\labl(a_1,a_2,C_j)$}
    \If{$|a_1|+|a_2|< n$ or $\min\{|a_1|,|a_2|\}<\lceil j/2\rceil-1$}
        \State{return $\mathsf{error}$} \Comment{not enough labels for $C_j$}
    \Else
        \State{label any path in $C_j$ of length $\min\{|a_2|,\lfloor j/2\rfloor+1\}$ by a suffix of $a_2$}
        \State{label the remaining path in $C_j$ by a suffix of $a_1$}
    \EndIf
\end{algorithmic}
\end{algorithm*}

The definition says that we label with suffixes (rather than arbitrary substrings) of arcs and use the longest possible suffix of the \emph{second} arc. By default, we suppose that both suffixes can be read on the cycle in \emph{the same direction} (i.e., the last labels from $a_1$ and $a_2$ are at the distance $\approx j/2$). 

\begin{wrapfigure}[8]{R}{0.68\textwidth} 
    \vspace*{-5.5mm}
    \includegraphics[scale=0.6, trim = 32 722 173 32, clip]{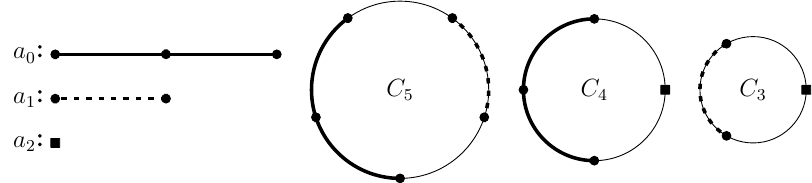}
    \caption{Example: 2-arc labeling for the family $\cC_5$.}
    \label{f:2arc}
    \vspace*{-1mm}
\end{wrapfigure}
The 2-arc scheme starts with a set $A$ of $m$ pairwise disjoint arcs of sufficient length. 

It calls the function $\labl()$ for each pair of arcs in $A$. 
As a result, it produces a family of up to $\frac{m(m+1)}{2}$ labeled cycles. 
Lemma~\ref{l:2arcvalid} below proves that the result is a labeling; we call it the \emph{2-arc labeling}.
In Fig.~\ref{f:2arc}, such a labeling of the family $\cC_5$ is shown.

\begin{lemma} \label{l:2arcvalid}
    The 2-arc labeling scheme is valid and produces arc labelings. 
\end{lemma}
\begin{proof}
    The intersection of two cycles is an arc (possibly, empty) by construction, so it suffices to prove that the output of the 2-arc scheme is a labeling.
    Thus we need to define the function $d'$ on labels.
    This is possible iff for every two labels $u,v$ the distance $d(u,v)$ is the same for each cycle containing both $u$ and $v$. 
    Since the scheme uses each pair of arcs once, the labels $u,v$ sharing several cycles belong to some arc $a\in A$.
    The intersection of a cycle $C_i$ with $a$ is at most the diameter of $C_i$.
    Hence $d(u,v)$ in $C_i$ is the same as $d(u,v)$ in $a$, and this property holds for any cycle shared by $u$ and $v$. Thus the scheme is valid. \qed
\end{proof}


\begin{remark} \label{r:2arcnum}
    For a 2-arc labeling of the family $\cC_n$ one can take $\sqrt{2n}+O(1)$ arcs of length $\frac{n}{2}+O(1)$ each, to the total of $\frac{n\sqrt{n}}{\sqrt{2}}+O(n)$ labels. 
    So the 2-arc labeling beats the folklore labeling, which requires $\frac{3n\sqrt{n}}{4}+O(n)$ labels by Proposition~\ref{p:folk_lab}.
\end{remark}

Next, we develop the idea of 2-arc labeling to obtain a scheme that, as we believe, produces asymptotically optimal labelings for the families $\cC_n$.





\paragraph{Chain labeling scheme.} First we present the 2-arc labeling scheme for the family $\cC_n$ as greedy Algorithm~\ref{alg:2arc}, which labels cycles in the order of decreasing length and proceeds in \emph{phases} until all cycles are labeled. 
Each phase starts with creating a new arc with the function $\create(arc,length)$; then the function $\labl()$ is called in a cycle, each time using the new arc and one of the earlier created arcs. 
The length of the new arc is taken to barely pass the length condition for the first call to $\labl()$. 
In the preliminary phase 1 (lines 1--3), two arcs are created and the largest cycle is labeled. Other phases are iterations of the \textbf{while} cycle in lines 4--10. See Fig.~\ref{f:2arc} for the example.

\vspace*{-2mm}
\begin{algorithm*}[!htb]
\caption{: Greedy 2-arc labeling scheme for $\cC_n$} 
\label{alg:2arc}
\begin{algorithmic}[1]
    \State{$\create(a_0,\lceil n/2\rceil)$; $\create(a_1,\lfloor n/2\rfloor)$}
    \State{$\labl(a_0,a_1,C_n)$}
    \State{$i\gets 2$; $j\gets n-1$} \Comment{next arc to create; next cycle to label}
    \While{$j>2$} \Comment{start of $i$th phase}
        \State{$\create(a_i, \lceil j/2\rceil-1)$} \Comment{minimum length of an arc needed to label $C_j$}
        \State{$k\gets 0$} \Comment{next arc to use}
        \While{$k<i$ and $j>2$}
            \State{$\labl(a_k,a_i, C_j)$}
            \State{$j\gets j-1$; $k\gets k+1$}
        \EndWhile
        \State{$i\gets i+1$}
    \EndWhile
\end{algorithmic}
\end{algorithm*}
\vspace*{-2mm}

The \emph{chain scheme} is a modification of Algorithm~\ref{alg:2arc} that allows to use previously created arcs more efficiently. 
The difference is in the first parameter of the $\labl$ function (line 8). 
During phase $i$, a \emph{chain} is a path labeled by the concatenation $a_0 a_1\cdots a_{i-1}$ of the strings labeling all previously created arcs. 
Though formally the chain is an arc, the distance between labels in the chain may differ from the distance between the same labels in already labeled cycles. 
For example, the string $a_0a_1$ labels both the cycle $C_n$ with the diameter $\lfloor n/2\rfloor+1$ and a path of diameter $n-1$.
However, with some precautions the chain can be used for labeling cycles. 
The chain scheme is presented below as Algorithm~\ref{alg:chain}.
The auxiliary function $\trim(c)$ deletes the suffix of the chain $c$ that was used to label a cycle on the current iteration of the internal cycle.

\vspace*{-2mm}
\begin{algorithm*}[!htb]
\caption{: Chain labeling scheme for $\cC_n$} 
\label{alg:chain}
\begin{algorithmic}[1]
    \State{$\create(a_0,\lceil n/2\rceil)$; $\create(a_1,\lfloor n/2\rfloor)$}
    \State{$\labl(a_0,a_1,C_n)$}
    \State{$i\gets 2$; $j\gets n-1$} \Comment{next arc to create; next cycle to label}
    \While{$j>2$} \Comment{start of phase $i$}
        \State{$c\gets a_0a_1\cdots a_{i-1}$} \Comment{chain for phase $i$}
        \State{$\create(a_i, \lceil j/2\rceil-1)$} \Comment{minimum length of an arc needed to label $C_j$}
        \While{$|c|\ge \lceil j/2\rceil-1$ and $j>2$}
            \State{$\labl(c,a_i, C_j)$}
            \State{$j\gets j-1$; $\trim(c)$} \Comment{deleting the just used suffix from the chain}
        \EndWhile
        \State{$i\gets i+1$}
    \EndWhile
\end{algorithmic}
\end{algorithm*}
\vspace*{-2mm}

To prove validity of the chain scheme, we need an auxiliary lemma.

\begin{lemma} \label{l:chainarcs}
    Let $j$ be the length of the longest unlabeled cycle at the beginning of $i$'th phase of Algorithm~\ref{alg:chain}, $i\ge 2$. Then every substring of length $\le \lfloor j/2\rfloor +1$ of the chain $c=a_0a_1\cdots a_{i-1}$ labels an arc in a cycle $C_{j'}$ for some $j'>j$. 
\end{lemma}
\begin{proof}
    Note that if a string labels an arc in an already labeled cycle, then every its substring does the same. 
    We proceed by induction on $i$. 
    From line 2 we see that each substring of length $\lfloor n/2\rfloor+1$ of the string $a_0a_1$ labels a diameter in $C_n$. Hence we have the base case $i=2$. 
    Since $j$ decreases with each phase, the inductive hypothesis implies that at the start of $(i{-}1)$th phase each substring of length $\lfloor j/2\rfloor +1$ of $a_0a_1\cdots a_{i-2}$ labels an arc in some already labeled cycle. 
    Let $C_{\hat{j}}$ be the first cycle labeled at this phase. 
    Then a diameter of $C_{\hat{j}}$ is labeled with a suffix of $a_0a_1\cdots a_{i-2}$, say, $a'$, and the remaining arc is labeled with the whole string $a_{i-1}$.
    Since $\hat{j}>j$, both the prefix $a_0a_1\cdots a_{i-2}$ and the suffix $a'a_{i-1}$ of the chain $c$ have the desired property: each substring of length $\le \lfloor j/2\rfloor +1$ labels an arc in an already labeled cycle.
    As these prefix and suffix of $c$ intersect by a substring $a'$ of length $\ge \lfloor j/2\rfloor +1$, the whole chain $c$ has this property. 
    This proves the step case and the lemma. \qed 
\end{proof}

\begin{lemma}\label{l:chainvalid}
    The chain labeling scheme is valid and builds an arc labeling. 
\end{lemma}
\begin{proof}
    To prove that Algorithm~\ref{alg:chain} builds a labeling, it suffices to check the following property: in every cycle $C_j$, each pair of labels either do not appear in larger cycles or appear in some larger cycle at the same distance as in $C_j$. 
    This property trivially holds for $C_n$, so let us consider a cycle $C_j$ labeled at $i$'th phase, $i\ge2$.
    The cycle $C_j$ is labeled by two arcs: a substring of the chain $c=a_0\cdots a_{i-1}$ and a suffix of the new arc $a_i$.
    We denote them by $c'$ and $a_i'$ respectively.
    
    Suppose that a pair of labels $(u,v)$ from $C_j$ appeared in a larger cycle.
    Since all substrings of $c$ used for labeling together with $a_i$ are disjoint, $u$ and $v$ belong to the same arc ($c'$ or $a_i'$).
    Then the shortest $(u,v)$-path in $C_j$ is within this arc. 
    If $u,v$ are in $a_i'$, then $d(u,v)$ in $C_j$ is the same as in the larger cycle containing $a_i$. 
    If $u,v$ are in $c'$, then $d(u,v)$ in $C_j$ is the same as in the larger cycle containing the arc $c'$; as $|c'|\le \lceil j/2\rceil-1$, such a cycle exists by Lemma~\ref{l:chainarcs}. Hence we proved that Algorithm~\ref{alg:chain} indeed builds a labeling; it is an arc labeling by construction. \qed
\end{proof}

We call the labelings obtained by chain scheme \emph{chain labelings}. Now we estimate the efficiency of the scheme.

\begin{theorem}\label{t:chain}
    A chain labeling of a family $\cC_n$ of cycles uses $\frac{n\sqrt{n}}{\sqrt{6}}+O(n)$ labels.
\end{theorem}

We first need a simpler estimate.
\begin{lemma} \label{l:chainphases}
    Algorithm~\ref{alg:chain} labels $\cC_n$ using $O(\sqrt{n})$ phases and $O(n\sqrt{n})$ labels.
\end{lemma}
\begin{proof}
    Let us compare the runs of Algorithms~\ref{alg:2arc} and~\ref{alg:chain} on the family $\cC_n$. 
    Suppose that $\ell_i$ is the length of $a_i$ for Algorithm~\ref{alg:chain} and $N_i$ is the number of cycles labeled by Algorithm~\ref{alg:chain} during the first $i$ phases.
    For Algorithm~\ref{alg:2arc}, we denote the same parameters by $\ell_i'$ and $N_i'$. Note that $N_1=N_1'=1$.

    Let $i\ge2$. If $N_{i-1}=N_{i-1}'$, then $\ell_i=\ell_i'$ as both algorithms begin the $i$'th phase with the same value of $j$. 
    During this phase, Algorithm~\ref{alg:2arc} labels $i$ cycles, while Algorithm~\ref{alg:chain} labels at least $i$ cycles (the length of the chain allows this), and possibly more. 
    Hence $N_i\ge N_i'$.
    If $N_{i-1}>N_{i-1}'$, Algorithm~\ref{alg:chain} begins the $i$'th phase with smaller value of $j$ compared to Algorithm~\ref{alg:2arc}. 
    Then $\ell_i\le \ell_i'$ and again, the length of the chain allows Algorithm~\ref{alg:chain} to label at least $i$ cycles during the $i$'th phase.
    Hence $N_i>N_i'$ (or $N_i=N_i'=n-2$ if both algorithms completed the labeling during this phase).
    Therefore,  Algorithm~\ref{alg:chain} uses at most the same number of phases and at most the same number of labels as Algorithm~\ref{alg:2arc}. The latter uses $\sqrt{2n}+O(1)$ arcs and thus $O(n\sqrt{n})$ labels. The lemma is proved. \qed
\end{proof}

\begin{proof}[of Theorem~\ref{t:chain}]
    Let us count distinct \emph{pairs} of labels. 
    First we count the pairs of labels that appear together in some cycle.
    We scan the cycles in the order of decreasing length and add ``new'' pairs (those not appearing in larger cycles) to the total. 
    After applying Algorithm~\ref{alg:chain} to $\cC_n$, each cycle $C_j$ is labeled by two arcs of almost equal length: one of them is a substring of the chain $c$ and the other one is a suffix of the current arc $a_i$.
    All pairs from $c$ in $C_j$ are not new by Lemma~\ref{l:chainarcs}.
    All pairs between $c$ and $a_i$ are new by construction, and their number is $\frac{j^2}{4}+O(j)$. Over all cycles $C_j$, this gives $\frac{n^3}{12}+O(n^2)$ distinct pairs in total.
    All pairs from $a_i$ are new if and only if $C_j$ is the first cycle labeled in a phase. 
    By Lemma~\ref{l:chainphases}, Algorithm~\ref{alg:chain} spends $O(\sqrt{n})$ phases.
    Hence there are $O(\sqrt{n})$ arcs $a_i$, each containing $O(n^2)$ pairs of labels, for the total of $O(n^{5/2})$ pairs.
    Therefore, the number of pairs that appear together in a cycle is $\frac{n^3}{12}+O(n^{5/2})$. 

    Next we count the pairs of labels that \emph{do not} appear together. 
    The labels in such a pair belong to different arcs.
    Let $u,v$ be from $a_{i'}$ and $a_i$ respectively, $i'<i$.
    If $u$ and $v$ \emph{appear} together, then the largest cycle containing both $u$ and $v$ was labeled at phase $i$. 
    Indeed, earlier phases have no access to $v$, and if $u$ and $v$ share a cycle labeled at a later phase, then Lemma~\ref{l:chainarcs} guarantees that they also appear in a larger cycle.
    Thus, to get the total number of pairs that do not appear together we count, for each phase $i$, the pairs $(u,v)$ such that $u$ is from the chain, $v$ is from $a_i$, and neither of the cycles labeled  during this phase contains both $u$ and $v$.
     
    There are three reasons why neither of the cycles labeled during phase $i$ contains both $u$ from the chain $c$ and $v$ from $a_i$. 
    First, this can be the last phase, which is too short to use $u$. 
    Since $|a_i|<n$ and $|c|=O(n\sqrt{n})$ by Lemma~\ref{l:chainphases}, the last phase affects $O(n^{5/2})$ pairs.
    Second, $u$ can belong to a short prefix of $c$ that remained unused during the phase due to the condition in line~7 of Algorithm~\ref{alg:chain}. 
    This prefix is shorter than $a_i$, so this situation affects less than $|a_i|^2$ pairs.
    As the number of phases is $O(\sqrt{n})$ (Lemma~\ref{l:chainphases}), the total number of such pairs is $O(n^{5/2})$.
    Third, $v$ can belong to a prefix of $a_i$ that was unused for the cycle containing $u$. 
    The number of labels from $a_i$ that were unused for  \emph{at least one} cycle during phase $i$ does not exceed $|a_i|-|a_{i+1}|$, which gives $O(n)$ labels over all phases. 
    Each such label $v$ is responsible for $O(n\sqrt{n})$ pairs by Lemma~\ref{l:chainphases}, for the total of $O(n^{5/2})$ pairs. Thus there are $O(n^{5/2})$ pairs that do not appear together.

    Putting everything together, we obtain that a chain labeling of $\cC_n$ contains $p=\frac{n^3}{12}+O(n^{5/2})$ pairs of labels. Hence the number $ch(n)$ of labels is
\begin{equation*}
    ch(n)= \sqrt{2p}+O(1)= \frac{n\sqrt{n}}{\sqrt{6}}\cdot \sqrt{1+O(n^{-1/2})}+O(1)=\frac{n\sqrt{n}}{\sqrt{6}} +O(n),
\end{equation*}
as required. \qed
\end{proof}

\section{Chain Labelings vs Optimal Labelings}

The chain labeling beats the folklore labeling almost twice in the number of labels (Theorem~\ref{t:chain} versus Proposition~\ref{p:folk_lab}). 
However, it is not clear how good this new labeling is, given that the known lower bound (Proposition~\ref{p:lower}) looks rather weak. 
In this section we describe the results of an experimental study we conducted to justify Conjecture~\ref{con:opt}, stating that the chain labeling is asymptotically optimal.

\begin{conjecture} \label{con:opt}
    $\lambda(n)=\frac{n\sqrt{n}}{\sqrt{6}}+O(n)$.
\end{conjecture}

We proceed in three steps, which logically follows each other.

\paragraph{Step 1.} Compute as many values of $\lambda(n)$ as possible and compare them to $\frac{n\sqrt{n}}{\sqrt{6}}$.


\paragraph{Outline of the search algorithm.} 
To compute $\lambda(n)$, we run a recursive depth-first search, labeling cycles in the order of decreasing length. 
The upper bound $\mathit{max}$ on the total number of labels is a global variable.
The recursive function $\mathsf{labelCycle}(j,L,D)$ gets the length $j$ of the cycle to label, the set $L$ of existing labels, and the table $D$ of known distances between them. 
The function runs an optimized search over all subsets of $L$.
When it finds a subset $X$ that is both
\begin{itemize}
    \item \emph{compatible}: all labels from $X$ can be assigned to the nodes of $C_j$ respecting the distances from $D$, and
    \item \emph{large}: labeling $C_j$ with $X\cup Y$, where the set $Y$ of labels is disjoint with $L$, holds the total number of labels below the upper bound $\mathit{max}$,
\end{itemize}
it labels $C_j$ with $X\cup Y$, adds $Y$ to $L$ to get some $L'$, adds newly defined distances to $D$ to get some $D'$, and compares $j$ to 3. 
If $j=3$, the function reports $(L',D')$, sets $\mathit{max}=\#L'$, and returns; otherwise, it calls $\mathsf{labelCycle}(j{-}1,L',D')$. The value of $\mathit{max}$ in the end of search is reported as $\lambda(n)$.

\begin{wrapfigure}[23]{R}{0.6\textwidth} 
    \vspace*{-2mm}
    \includegraphics[scale=0.5]{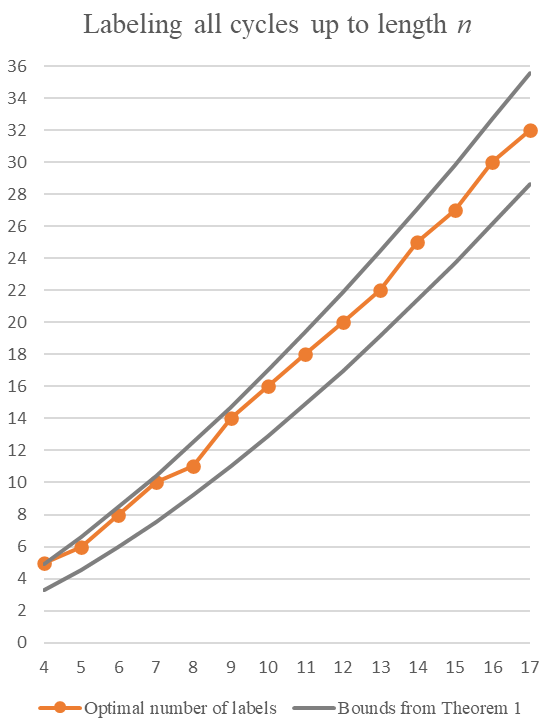}
    \vspace*{-3mm}
    \caption{Bound from Theorem~\ref{t:chain} versus the optimal numbers $\lambda(n)$ of labels. The lower and upper bounds are $\frac{n\sqrt{n}}{\sqrt{6}}$ and $\frac{n(\sqrt{n}+1)}{\sqrt{6}}$ respectively.}
    \label{f:optimal}
\end{wrapfigure}
\paragraph{Results.} We managed to find $\lambda(n)$ for $n\le 17$; for $n=17$ the algorithm made over $5\cdot10^{12}$ recursive calls.
Computing $\lambda(18)$ would probably require a cluster. 
The witness labelings can be found in the Appendix; the numbers $\lambda(n)$ fit well in between the bounds $\frac{n\sqrt{n}}{\sqrt{6}}$ and $\frac{n(\sqrt{n}+1)}{\sqrt{6}}$ (see Fig.~\ref{f:optimal}). 
As a side result we note that almost all optimal labelings we discovered are \emph{arc labelings}.

If we view the ``corridor'' in Fig.~\ref{f:optimal} as an extrapolation for $\lambda(n)$ for big $n$, we have to refer $ch(n)$ to this corridor.

\paragraph{Step 2.} Estimate the constant in the $O(n)$ term in Theorem~\ref{t:chain}, to compare $ch(n)$ to the results of step 1. 

\paragraph{Results.} We computed $ch(n)$ for many values of $n$ in the range $[10^3..10^7]$. In all cases,  $ch(n)\approx\frac{n(\sqrt{n}+1.5)}{\sqrt{6}}$, which is only $\frac{n}{2\sqrt{6}}$ away from the ``corridor'' in Fig.~\ref{f:optimal}.

The natural next question is whether we can do better. 

\paragraph{Step 3.} Find resources to improve the chain scheme to come closer to $\lambda(n)$.

In order to reduce the amount of resources ``wasted'' by the chain scheme, we describe three improving tricks. 
An example of their use is an optimal labeling of $\cC_{14}$ presented in Fig.~\ref{f:14}.

\paragraph{Trick 1: reusing ends of arcs.} During a phase, if a cycle is labeled with the strings $a_1\cdots a_j$ from the new arc and $c_x\cdots c_{x+j}$ or $c_x\cdots c_{x+j+1}$ from the chain, then it is correct to use for the next cycle the string $c_{x-j+1}\cdots c_x$ (resp., $c_{x-j}\cdots c_x$), thus reusing the label $c_x$; for example see $C_{13}$ and $C_{12}$ in Fig.~\ref{f:14}.

\begin{figure}[!htb]
    \centering
    \includegraphics[scale=0.77, trim = 25 610 174 20, clip]{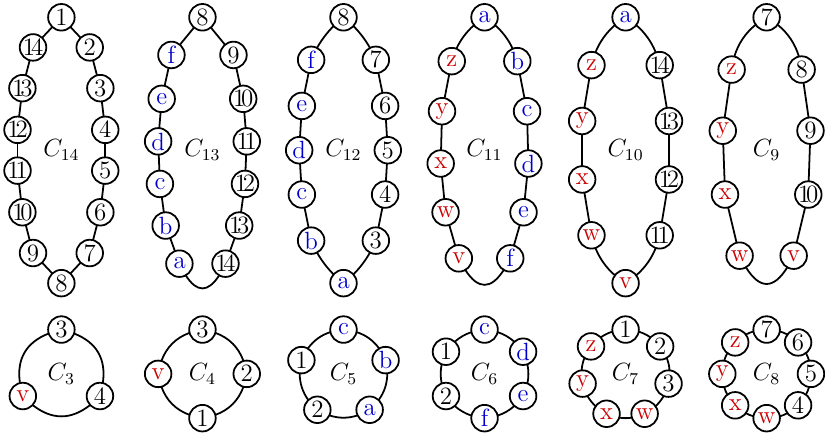}
    \vspace*{-2mm}
    \caption{An optimal labeling of $\cC_{14}$ by the enhanced chain scheme.}
    \vspace*{-2mm}
    \label{f:14}
\end{figure}

The function $\labl()$ is defined so that the above situation happens only in the beginning of the phase, so this trick saves 1 or 2 labels in the chain. Still, sometimes this  leads to labeling an additional cycle during a phase. 

\paragraph{Trick 2: using chain remainders.} In the end of a phase, we memorize the remainder $c$ of the chain and the current arc $a$. 
Thus, at any moment we have the set $S$ of such pairs of strings (initially empty). 
Now, before labeling a cycle we check whether $S$ contains a pair $(c,a)$ that can label this cycle. 
If yes, we extract $(c,a)$ from $S$ and run a ``mini-phase'', labeling successive cycles with $c$ as the arc and $a$ as the chain; when the mini-phase ends, with $a'$ being the chain remainder, we put the pair $(a',c)$ in $S$ and proceed to the next cycle.
Otherwise, we label the current cycle as usual. 
In Fig.~\ref{f:14}, the pair $(12,abcdef)$ was added to $S$ after phase 2.
Later, the cycles $C_6$ and $C_5$ were labeled during a mini-phase with this pair; note that trick 1 was used to label $C_5$.

\paragraph{Trick 3: two-pass phase.} We combine the last two phases as follows. 
Let $a$ be the arc in the penultimate phase and the prefix $a'$ of $a$ was unused for labeling the last cycle during the phase. 
As this cycle consumed $|a|-|a'|$ labels from the arc, for the last phase we need the arc of the same (up to a small constant) length. 
We create such an arc of the form $\hat aa'$, where the labels from $\hat a$ are new (if $|a'|>|a|/2$, no new labels are needed). 
During the last phase, we reverse both the arc and the chain: the chain $c$ is cut from the beginning, and the arc $\hat aa'$ is cut from the end. 
In this way, the labels from $a'$ will not meet the labels from $c$ for the second time, so the phase will finish correctly. 
An additional small trick can help sometimes: for a cycle, we use one less symbol from the arc than the maximum possible. This charges the chain by an additional label but this label can be reused from the previous cycle by employing trick 1. In Fig.~\ref{f:14}, this was done for $C_8$. As a result, $a'=v$ did not meet $1,\ldots,6$ from the chain and we were able to label $C_4$ and $C_3$ with no additional labels.

\vspace*{-2mm}
\paragraph{Results.} We applied the \emph{enhanced chain scheme}, which uses tricks 1--3, for labeling many families $\cC_n$ for $n\in [10^3..10^7]$. 
In all cases, we get the number of labels $ch^+(n)\approx \frac{n(\sqrt{n}+1)}{\sqrt{6}}$, which is exactly the upper bound of the corridor in Fig.~\ref{f:optimal}. 
In Table~\ref{tab:comparison}, we compare the results of chain schemes to the known optima.

\begin{table}[!htb]
\tabcolsep=3pt
\centering
\begin{tabular}{|c|c|c|c|c|}
\hline
$n$ & $ch(n)$ & $ch^+(n)$ & $ch^{++}(n)$& $\lambda(n)$ \\
\hline
7 & 10 & 10 & 10 & 10\\
8 & 13 & 12 & 11 & 11\\
9 & 14 & 14 & 14 & 14\\
10 & 17 & 17 & 16 & 16\\
11 & 19 & 18 & 18 & 18\\
12 & 21 & 21 & 20 & 20\\
\hline
\end{tabular} 
\hskip5mm
\begin{tabular}{|c|c|c|c|c|}
\hline
$n$ & $ch(n)$ & $ch^+(n)$ & $ch^{++}(n)$& $\lambda(n)$ \\
\hline
13 & 24 & 23 & 22 & 22\\
14 & 28 & 25 & 25 & 25\\
15 & 30 & 29 & 28 & 27\\
16 & 33 & 33 & 31 & 30\\
17 & 35 & 34 & 33 & 32\\
18 & 38 & 38 & 36 & 35?\\
\hline
\end{tabular} 
\vspace*{2mm}
\caption{Comparison of chain and optimal labelings for small $n$. The columns $ch(n)$, $ch^+(n)$, $\lambda(n)$ contain the number of labels for the chain scheme, enhanced chain scheme, and the optimal labeling, respectively; $ch^{++}(n)$ refers to the hybrid scheme: we perform the enhanced chain scheme till the two-pass phase, and replace this phase by a search for an optimal labeling of remaining cycles.}
\vspace*{-4mm}
\label{tab:comparison}
\end{table}

Overall, the results gathered in Fig.~\ref{f:optimal} and Table~\ref{tab:comparison} together with the behavior of $ch(n)$ and $ch^+(n)$ for big $n$ give a substantial support to Conjecture~\ref{con:opt}.

\section{Discussion and Future Work}

The main open problem for distance labeling of the families $\cC_n$ is the gap between the lower bound $\lambda(n)=\Omega(n\sqrt[3]{n})$ and the upper bound $\lambda(n)=O(n\sqrt{n})$. 
Our results suggest that the upper bound provides the correct asymptotics but improving the lower bound probably needs some new approach. 

This is pretty alike the situation with distance labeling of planar graphs. Here, the gap (in terms of the length of a label in bits, i.e., in logarithmic scale) is between $\Omega(\sqrt[3]{n})$ and $O(\sqrt{n})$ and there is an evidence \cite{AGMW18, GaUz23} that the upper bound gives the correct asymptotics but the existing approach does not allow to improve the lower bound. Another similar gap between the cubic root and the quadratic root appears in the adjacency labeling problem for $\cC_n$ \cite{AAHKS20}. 

As a possible approach to the improvement of the lower bound for $\lambda(n)$ we propose to study the number $\lambda_k(n)$ of labels needed to label the family $\cC_{n,k}=\{C_n,C_{n-1},\ldots,C_{n-k+1}\}$, starting from small $k$. Algorithm~\ref{alg:2arc} and Lemma~\ref{l:diam} imply $\lambda_2(n)=\lambda_3(n)=1.5n+O(1)$ but already the next step is not completely trivial.

\bibliographystyle{splncs04}
\bibliography{bib}

\newpage
\section*{Appendix}

\subsection*{Examples of optimal labelings of $\cC_n$ for small $n$}

Labels are nonnegative integers.

\tabcolsep=1.5pt
\noindent
{\footnotesize
\begin{tabular}{l*{17}{c}}
$C_{17}:$ & 0&1&2&3&4&5&6&7&8&9&10&11&12&13&14&15&16 \\
$C_{16}:$ & 0&1&2&3&4&5&6&7&8&17&18&19&20&21&22&23\\
$C_{15}:$ & 0&1&2&24&25&26&27&28&29&30&19&20&21&22&23 \\
$C_{14}:$ & 0&1&2&24&25&26&27&28&31&12&13&14&15&16 \\
$C_{13}:$ & 12&13&14&15&16&17&18&19&20&21&22&23&31 \\
$C_{12}:$ &9&10&11&12&31&23&22&21&20&19&30&29 \\
$C_{11}:$ &7&8&9&10&11&12&31&28&27&26&25 \\
$C_{10}:$ &7&8&17&18&30&29&28&27&26&25 \\
$C_{9}:$ &2&3&4&5&6&27&26&25&24 \\
$C_{8}:$ &3&4&5&6&27&28&29&30 \\
$C_{7}:$ &8&9&10&11&24&18&17 \\
$C_{6}:$ &15&16&17&18&30&29 \\
$C_{5}:$ &13&14&15&29&30 \\
$C_{4}:$ &3&4&5&31&&&&&&&\multicolumn{3}{c}{\normalsize$n=17$} \\
$C_{3}:$ &5&6&31\\
\end{tabular} }
{\footnotesize
\begin{tabular}{l*{8}{c}}
$C_{7}:$ &0&1&2&3&4&5&6 \\
$C_{6}:$ &0&1&2&3&7&8 \\
$C_{5}:$ &5&4&3&7&8 \\
$C_{4}:$ &5&6&7&8 \\
$C_{3}:$ &0&1&9&&\multicolumn{3}{c}{\normalsize$n=7$}\\
&\\
$C_{8}:$ &0&1&2&3&4&5&6&7 \\
$C_{7}:$ &0&1&2&3&8&9&10 \\
$C_{6}:$ &5&4&3&8&9&10 \\
$C_{5}:$ &5&6&7&9&10 \\
$C_{4}:$ &0&7&9&10 \\
$C_{3}:$ &6&7&8&&\multicolumn{3}{c}{\normalsize$n=8$}\\
\end{tabular} }

\noindent
{\footnotesize
\begin{tabular}{l*{16}{c}}
$C_{16}:$ & 0&1&2&3&4&5&6&7&8&9&10&11&12&13&14&15 \\
$C_{15}:$ & 0&16&17&18&19&20&21&22&23&24&25&26&27&28&29\\
$C_{14}:$ &0&1&2&3&4&5&6&7&21&20&19&18&17&16 \\
$C_{13}:$ &0&1&2&3&4&5&6&24&25&26&27&28&29 \\
$C_{12}:$ &0&15&14&13&12&11&10&20&19&18&17&16 \\
$C_{11}:$ &7&8&9&10&11&12&25&26&27&28&29 \\
$C_{10}:$ &0&15&14&13&12&25&26&27&28&29 \\
$C_{9}:$ &12&13&14&15&21&22&23&24&25 \\
$C_{8}:$ &9&10&11&12&25&24&23&22 \\
$C_{7}:$ &6&7&8&9&22&23&24 \\
$C_{6}:$ &1&2&3&4&22&23 \\
$C_{5}:$ &7&8&9&20&21 \\
$C_{4}:$ &4&5&23&22&&&&&&&\multicolumn{3}{c}{\normalsize$n=16$} \\
$C_{3}:$ &8&9&16\\
\end{tabular} }
{\footnotesize
\begin{tabular}{l*{9}{c}}
$C_{9}:$ &0&1&2&3&4&5&6&7&8 \\
$C_{8}:$ &0&1&2&3&4&9&10&11 \\
$C_{7}:$ &7&6&5&4&9&10&11 \\
$C_{6}:$ &0&1&2&3&12&13 \\
$C_{5}:$ &5&4&3&12&13 \\
$C_{4}:$ &5&6&7&13&&&\multicolumn{3}{c}{\normalsize$n=9$} \\
$C_{3}:$ &8&7&13\\
\end{tabular} }

\noindent
{\footnotesize
\begin{tabular}{l*{15}{c}}
$C_{15}:$ & 0&1&2&3&4&5&6&7&8&9&10&11&12&13&14\\
$C_{14}:$ &0&1&15&16&17&18&19&20&21&22&23&24&25&26 \\
$C_{13}:$ &0&14&13&12&11&10&9&21&22&23&24&25&26 \\
$C_{12}:$ &0&1&15&16&17&18&19&10&11&12&13&14 \\
$C_{11}:$ &2&3&4&5&6&15&16&17&18&19&20 \\
$C_{10}:$ &6&7&8&9&10&19&18&17&16&15 \\
$C_{9}:$ &5&6&7&8&9&21&22&23&24 \\
$C_{8}:$ &0&1&2&3&4&24&25&26 \\
$C_{7}:$ &2&3&4&23&22&21&20 \\
$C_{6}:$ &5&6&7&8&25&26 \\
$C_{5}:$ &9&10&19&20&21 \\
$C_{4}:$ &5&24&25&26&&&&&&&\multicolumn{3}{c}{\normalsize$n=15$} \\
$C_{3}:$ &4&5&24\\
\end{tabular} }
{\footnotesize
\begin{tabular}{l*{10}{c}}
$C_{10}:$ &0&1&2&3&4&5&6&7&8&9 \\
$C_{9}:$ &0&1&2&10&11&12&13&14&15 \\
$C_{8}:$ &0&9&8&7&6&13&14&15 \\
$C_{7}:$ &2&3&4&5&12&11&10 \\
$C_{6}:$ &3&4&5&6&13&14 \\
$C_{5}:$ &5&6&7&11&12 \\
$C_{4}:$ &3&4&5&15&&&&\multicolumn{3}{c}{\normalsize$n=10$} \\
$C_{3}:$ &3&14&15\\
\end{tabular} }

\noindent
{\footnotesize
\begin{tabular}{l*{14}{c}}
$C_{14}:$ &0&1&2&3&4&5&6&7&8&9&10&11&12&13\\
$C_{13}:$ &0&1&14&15&16&17&18&19&20&21&22&23&24 \\
$C_{12}:$ &0&1&2&3&4&5&6&20&21&22&23&24 \\
$C_{11}:$ &0&13&12&11&10&9&20&21&22&23&24 \\
$C_{10}:$ &0&1&14&15&16&17&10&11&12&13 \\
$C_{9}:$ &1&2&3&4&5&17&16&15&14 \\
$C_{8}:$ &5&6&7&8&9&19&18&17 \\
$C_{7}:$ &6&7&8&9&14&15&16 \\
$C_{6}:$ &3&4&5&17&18&19 \\
$C_{5}:$ &6&7&8&21&20 \\
$C_{4}:$ &2&3&19&18&&&&&&&\multicolumn{3}{c}{\normalsize$n=14$} \\
$C_{3}:$ &7&8&22\\
\end{tabular} }
{\footnotesize
\begin{tabular}{l*{11}{c}}
$C_{11}:$ &0&1&2&3&4&5&6&7&8&9&10 \\
$C_{10}:$ &0&1&2&11&12&13&14&15&16&17 \\
$C_{9}:$ &0&10&9&8&7&14&15&16&17 \\
$C_{8}:$ &2&3&4&5&6&13&12&11 \\
$C_{7}:$ &3&4&5&6&14&15&16 \\
$C_{6}:$ &6&7&8&11&12&13 \\
$C_{5}:$ &8&9&10&12&11 \\
$C_{4}:$ &3&4&5&17&&&&\multicolumn{3}{c}{\normalsize$n=11$} \\
$C_{3}:$ &3&16&17\\
\end{tabular} }

\noindent
{\footnotesize
\begin{tabular}{l*{13}{c}}
$C_{13}:$ &0&1&2&3&4&5&6&7&8&9&10&11&12 \\
$C_{12}:$ &0&1&2&3&4&5&13&14&15&16&17&18 \\
$C_{11}:$ &6&7&8&9&10&11&17&18&19&20&21 \\
$C_{10}:$ &6&13&14&15&16&17&18&19&20&21 \\
$C_{9}:$ &6&7&8&9&10&16&15&14&13 \\
$C_{8}:$ &0&1&2&3&19&20&21&12 \\
$C_{7}:$ &3&4&5&6&21&20&19 \\
$C_{6}:$ &10&11&12&14&15&16 \\
$C_{5}:$ &0&12&11&17&18 \\
$C_{4}:$ &10&11&17&16&&&&&&\multicolumn{3}{c}{\normalsize$n=13$} \\
$C_{3}:$ &5&6&13\\
\end{tabular} }
{\footnotesize
\begin{tabular}{l*{12}{c}}
$C_{12}:$ &0&1&2&3&4&5&6&7&8&9&10&11 \\
$C_{11}:$ &0&1&2&12&13&14&15&16&17&18&19 \\
$C_{10}:$ &0&11&10&9&8&7&16&17&18&19 \\
$C_{9}:$ &3&4&5&6&7&16&17&18&19 \\
$C_{8}:$ &2&3&4&5&6&14&13&12 \\
$C_{7}:$ &6&7&8&9&12&13&14 \\
$C_{6}:$ &9&10&11&14&13&12 \\
$C_{5}:$ &6&7&16&15&14 \\
$C_{4}:$ &3&4&5&15&&&&&\multicolumn{3}{c}{\normalsize$n=12$} \\
$C_{3}:$ &8&9&15\\
\end{tabular} }

\end{document}